\documentclass[11pt]{article} 
 \usepackage{dsfont}
 \usepackage{ams}
 \usepackage{amsmath}
 \usepackage{amsthm}
 \usepackage{txfonts,bm}
 \usepackage{latexsym}
 \usepackage{natbib}
  \def\citeapos#1{\citeauthor{#1}'s (\citeyear{#1})}
 \usepackage[colorlinks,citecolor=blue,urlcolor=blue,linkcolor=blue]{hyperref}

 \newtheorem{theo}{\small\bf Theorem}[section]
 \newtheorem{lem}{\small\bf Lemma}[section]
 \newtheorem{prop}{\small\bf Proposition}[section]
 \newtheorem{rem}{\small\bf Remark}[section]
 \newenvironment{REM}{\begin{rem} \rm}{\end{rem}}
 \newtheorem{defi}{\small\bf Definition}[section]
 
 \newtheorem{cor}{\small\bf Corollary}[section]

 \numberwithin{equation}{section}

 \newcommand{\be}{\begin{equation}}
 \newcommand{\ee}{\end{equation}}
 \newcommand{\E}{\operatorname{\mathds{E}}}
 \renewcommand{\Pr}{\operatorname{\mathds{P}}}
 \newcommand{\Var}{\operatorname{Var}}

 \newcommand{\bbb}[1]{\textrm{\boldmath $ #1 $}}

 \newcommand{\one}{\mathds{1}}
 \newcommand{\RR}{\mathds{R}}

 \newcommand{\CX}{\mathcal{X}}
 \newcommand{\CB}{\mathcal{B}}
 \newcommand{\CD}{\mathcal{D}}

 \newcommand{\dtv}{d_{\textrm{\rm \small tv}}}
 \newcommand{\ud}{{\rm d}}
 \newcommand{\Poi}{\operatorname{Poi}}

 \title{ \Large\bf A factorial moment distance and an application to the matching
 problem}

 \author{\large
 G.\ Afendras\footnote{Corresponding author. {\tt
 e-mail address:\ gafendra@buffalo.edu}}
 \ \ and  \ N.\ Papadatos\footnote{Work partially supported by the University of
 Athens' Research fund under Grant
 70/4/5637.}}
 \date{\normalsize
 Departments of Biostatistics and Biomedical Informatics, SUNY at Buffalo, 3435 Main Street, 726 Kimball Hall, School of Public
Health and Health Professions, and School of Medicine, Buffalo, NY 14260
 \\
 and
 \\
 Department of Mathematics,
 University of Athens, \\
 Panepistemiopolis, 157 84 Athens, Greece.
 }

\begin{document}

 \maketitle

 \thispagestyle{empty}

 \begin{abstract}
 \noindent
 In this note we introduce the notion of factorial moment distance for non-negative
 integer-valued random variables and we compare it with the total variation distance.
 Furthermore, we study the rate of convergence in the classical matching problem and in a
 generalized matching distribution.
 \end{abstract}
 {\footnotesize {\it MSC}:  Primary 60E05, 60E15; Secondary 44A10, 41A25.
 }

 \section{Introduction}
 \label{sec:intro}

 Let $\bbb{\pi}_n=\big(\pi_n(1),\ldots,\pi_n(n)\big)$ be a random permutation of
 $T_n=\{1,2,\ldots,n\}$, in the sense that
 $\bbb{\pi}_n$ is uniformly distributed over the $n!$ permutations of $T_n$.
 A number $j$ is a fixed point of $\bbb{\pi}_n$ if $\pi_n(j)=j$.
 Denote by $Z_n$ the total
 number of fixed points of $\bbb{\pi}_n$,
 \[
 Z_n=\sum_{j=1}^{n}\mathds{1}\{\pi_n(j)=j\},
 \]
 where $\one$ stands for the indicator function.
 The study of $Z_n$ corresponds to the famous {\it matching problem},
 introduced by Montmort in 1708.
 Obviously, $Z_n$ can take the values
 $0,1,\ldots,n-2,n$, and its
 exact
 distribution, using standard combinatorial arguments,
 is found to be
 \[
 \Pr(Z_n=j)=\frac{1}{j!}\sum_{k=0}^{n-j}\frac{(-1)^k}{k!},\quad j=0,1,\ldots,n-2,n.
 \]
 It is obvious that $Z_n$ converges in law to $Z$, where $Z$
 is the standard Poisson distribution, $\Poi(1)$. Furthermore,
 the Poisson approximation is very accurate even for small $n$
 \citep[evidence of this may be found in][]{BHJ1992}.
 Bounds on the error of the Poisson approximation in the matching
 problem, especially concerning the total variation distance, are also
 well-known.
 Recall that the total variation distance
 of any two rv's $X_1$ and $X_2$ is defined as
 \[
 \dtv(X_1,X_2)=\sup_{A\in\CB(\RR)}\big|\Pr(X_1\in A)-\Pr(X_2\in A)\big|,
 \]
 where $\CB(\RR)$ is the Borel $\sigma$-algebra of $\RR$.
 An appealing result is given by
 \citet{Dia1987},
 who proved that $\dtv(Z_n,Z)\le\frac{2^n}{n!}$.
 This bound has been improved by
 \citet{DasGupta1999,DasGupta2005}:
 \begin{equation}
  \label{eq:tv-DG}
 \dtv(Z_n,Z)\le\frac{2^n}{(n+1)!}.
 \end{equation}
 It can be seen that $\dtv(Z_n,Z)\sim\frac{2^n}{(n+1)!}$, where
 $a_n\sim b_n$ means that $\lim_n \frac{a_n}{b_n}=1$; for a proof of a more general result
 see Theorem \ref{theo:order tv}, below.
 Therefore, the bound \eqref{eq:tv-DG} is of the correct order.

 Consider now the sets of discrete rv's
 \[
 \CD_n:=\{X : \Pr(X\in\{0,1,\ldots,n\})=1\},
 \ \ \
 \CD_\infty:=\{X : \Pr(X\in\{0,1,\ldots\})=1\}.
 \]
 Since the first $n$ moments of $Z_n$ and $Z$ are identical and
 $Z_n\in \CD_n$, $Z\in\CD_{\infty}$, one
 might think that
 \begin{equation}
 \label{eq:inf_X in Dn}
 \inf_{X\in\CD_n}\{\dtv(X,Z)\}\sim \dtv(Z_n,Z)\sim \frac{2^n}{(n+1)!}.
 \end{equation}
 However, \eqref{eq:inf_X in Dn} is not true. In fact,
 \begin{equation}
 \label{min}
 \min_{X\in\CD_n}\{\dtv(X,Z)\}=1-e^{-1}\sum_{j=0}^n \frac{1}{j!}\sim
 \frac{e^{-1}}{(n+1)!}.
 \end{equation}
 Indeed, for any $X_1,X_2\in\CD_\infty$ with probability mass
 functions (pmf's) $p_1$ and $p_2$, the
 total variation distance can be expresed as
 \begin{equation}
 \label{eq:rtv discrete}
 \dtv(X_1,X_2)=\frac{1}{2}\sum_{j=0}^{\infty}|p_1(j)-p_2(j)|
 =\sum_{j=0}^{\infty}\big(p_1(j)-p_2(j)\big)^+,
 \end{equation}
 where $x^+=\max\{x,0\}$. Thus, for any $X_1\in \CD_n$ (so that $p_1(j)=0$ for all $j>n$),
 we get
 \[
 \dtv(X_1,X_2)=\sum_{j=0}^n \big(p_1(j)-p_2(j)\big)^+\geq \sum_{j=0}^n \big(p_1(j)-p_2(j)\big)
 =1-\sum_{j=0}^n p_2(j)=\Pr(X_2>n),
 \]
 with equality if and only if $p_1(j)\geq p_2(j)$,  $j=0,1,\ldots,n$. Applying the
 preceding inequality to $p_2(j)=\Pr(Z=j)=\frac{e^{-1}}{j!}$ we get the equality in
 \eqref{min}, and
 the minimum is attained by any rv
 $X\in \CD_n$ with $\Pr(X=j)\geq \frac{e^{-1}}{j!}$, $j=0,1,\ldots,n$.
 Furthermore, the well-known Cauchy remainder in the Taylor expansion
 reads as
 \begin{equation}
 \label{taylor}
 f(x)-\sum_{j=0}^n \frac{f^{(j)}(0)}{j!}x^j
 =\frac{1}{n!}\int_{0}^x (x-y)^n f^{(n+1)}(y) \ud y.
 \end{equation}
 Applying \eqref{taylor} to $f(x)=e^x$ we get the expression
 \[
 1-e^{-1}\sum_{j=0}^n \frac{1}{j!}=e^{-1}\left(e-\sum_{j=0}^n \frac{1}{j!}\right)
 =\frac{e^{-1}}{n!}\int_0^1 (1-y)^n e^y \ud y,
 \]
 and by the obvious inequalities $1<e^y<1+(e-1)y$, $0<y<1$, we have
 \[
 \frac{1}{n+1}<\int_0^1 (1-y)^n e^y \ud y<\frac{1}{n+1}
 \left(1+\frac{e-1}{n+2}\right).
 \]
 It follows that
 \[
 \frac{e^{-1}}{(n+1)!}
 < \min_{X\in\CD_n}\{\dtv(X,Z)\}=1-e^{-1}\sum_{j=0}^n  \frac{1}{j!}
 < \frac{e^{-1}}{(n+1)!}\left(1+\frac{e-1}{n+2}\right),
 \]
 and therefore,
 $\min_{X\in\CD_n}\{\dtv(X,Z)\}\sim \frac{e^{-1}}{(n+1)!}$.

 In the present note we introduce and study a class of factorial moment distances,
 $\{d_\alpha, \alpha>0\}$. These metrics are designed to capture the discepancy
 among discrete distributions with finite moment generating function in a neiborhood of
 zero and, in addition, they
 satisfy the desirable property
 $\min_{X\in\CD_n}\{d_\alpha (X,Z)\}=d_\alpha(Z_n,Z)$.
 In Section
 \ref{sec:matching,generalized} we study the rate of convergence
 in a generalized matching problem,
 and we present closed form
 expansions and sharp inequalities
 for the factorial moment distance and the variational distance.

 \section{The factorial moment distance}
 \label{sec:fm distance}
 We start with the following observation:
 For the rv's $Z$ and $Z_n$,
 \begin{equation}
 \label{eq:E(Z)_k,E(Zn)_k}
 \E(Z)_k=1 \quad\textrm{and}\quad \E(Z_n)_k=\one_{\{k\le n\}},\quad k=0,1,\ldots,
 \end{equation}
 where $\E(X)_k$ denotes the $k$-th order descending factorial moment of $X$
 [for each $x\in\RR$, $(x)_0=1$ and $(x)_k=x(x-1)\cdots(x-k+1)$, $k=1,2\ldots$].
 For a proof of a more general result see Lemma \ref{lem.3.1}, below.

 The factorial moment distance will be defined in a suitable sub-class
 of discrete random variables, as follows.
 For each $t\ge0$ we define
 \begin{equation}
 \label{eq:CX(t)}
 \CX(t) := \{\textrm{$X\in\CD_\infty :$ there exists $t'>t$ such that
  $P_X(1+t')<\infty$}\},
 \end{equation}
 where $P_X(u)=\E u^{X}$ is the probability generating function of $X$.
 Also, we define
 \begin{equation}
 \label{eq:CX(infty)}
 \CX(\infty) := \bigcap_{t\in[0,\infty)}\CX(t)=\{\textrm{$X\in\CD_\infty :$
 $P_X(1+t')<\infty$ for any $t'>0$}\}.
 \end{equation}
 Note that if $X\in\CD_n$ for some $n$ then $X\in\CX(t)$ for each $t\in[0,\infty]$;
 therefore, each $\CX(t)$ is non-empty.
 For $0\le t_1<t_2\le\infty$ it is obvious that $\CX(t_2)\subset\CX(t_1)$;
 that is, the family $\{\CX(t), 0\leq t\leq \infty\}$ is decreasing in $t$.

 If $X\in\CX(0)$ then there exists a $t'>0$ such that $P_X(1+t')<\infty$, i.e.,
 $\E e^{\theta X}<\infty$ where $\theta=\ln(1+t')>0$. Since $X$ is non-negative,
 $\E e^{\theta X}<\infty$ implies that $\E e^{uX}<\infty$ for all $u\in(-\theta,\theta)$,
 which means that $X$ has finite moment generating function at a neighborhood of zero.
 Therefore, $X$ has finite moments of any order and its pmf is characterized by its moments;
 equivalently, $X$ has finite descending factorial moment of any order and its pmf
 is characterized
 by these moments.
 This enables the following
 \begin{defi}
 \label{def:fm distance,convergence}
 {\rm
 (a) Let $X_1,X_2\in\CX(0)$. For $\alpha>0$ we define the
 {\it factorial moment distance of order $\alpha$} of $X_1,X_2$ by
 \begin{equation}
 \label{eq:fm distance}
 d_\alpha(X_1,X_2):=\sum_{k=1}^{\infty}\frac{\alpha^{k-1}}{k!}\big|\E(X_1)_k-\E(X_2)_k\big|.
 \end{equation}
 (b) Let $X\in\CX(0)$ and $\{X_n\}_{n=1}^{\infty}\subset\CX(0)$. We say that $X_n$
 {\it converges in factorial moment distance of order $\alpha$} to $X$,
 in symbols $X_n\to_\alpha X$,
 if
 \[
 d_\alpha(X_n,X)\to0,\ \textrm{as} \ n\to\infty.
 \]
 }
 \end{defi}
 One can easily check that the function
 $d_\alpha:\CX(0)\times\CX(0)\to[0,\infty]$ is a distance.
 Obviously, $X_n\to_\alpha X$ implies that the moments of $X_n$ converge
 to the corresponding moments of $X$. Since every $X\in\CX(0)$ is characterized
 by its moments, it follows that the $d_\alpha$ convergence (for any $\alpha>0$)
 is stronger than the convergence in law; the later is equivalent to the convergence
 in total variation -- see
  \citet{Wang1991}.
 Of course, the converse is not true even in $\CX(\infty)$.
 For example, consider the rv $X$ with $\Pr(X=0)=1$,
 and the sequence of rv's
 $\{X_n\}_{n=1}^{\infty}$, where each $X_n$ has pmf
 \[
 p_n(j)=\left\{\begin{array}{lcl}
          1-\frac{1}{n}        & , & j=0,\\
          \frac{1}{n}          & , & j=n.
        \end{array}\right.
 \]
 It is obvious that $\{X,X_1,X_2,\ldots\}\subset\CX(\infty)$, and
 the total variation distance is
 \[
 \dtv(X_n,X)=\frac{1}{n}\to0,\quad\textrm{as} \ n\to\infty.
 \]
 Moreover, since $\E(X)_k=0$ and $\E(X_n)_k=(n-1)_{k-1}\one\{k\le{n}\}$ for
 all $k=1,2,\ldots$, the $d_\alpha$ distance does not converge to zero:
 \[
 d_\alpha(X_n,X)=\sum_{k=1}^{\infty}\frac{\alpha^{k-1}}{k!}
 (n-1)_{k-1}\one{\{k\le{n}\}}>\frac{\alpha}{2}(n-1)\one{\{2\le{n}\}}\to\infty,
 \quad\textrm{as} \ n\to\infty.
 \]

 %
 %
 %

 \begin{REM}
 \label{rem:inf}
 Let $X\in\CD_n\smallsetminus\{Z_n\}$. It is obvious that $\E(X)_k=0$ for all $k>n$, and
 we can find an index $k\in\{1,2,\ldots,n\}$ such that $\E(X)_k\ne1$. From
 \eqref{eq:E(Z)_k,E(Zn)_k}
 and \eqref{eq:fm distance} we see that $d_\alpha(X,Z)>d_\alpha(Z_n,Z)$.
 Hence,
 \[
 \inf_{X\in\CD_n}\{d_\alpha(X,Z)\}=d_\alpha (Z_n,Z),\quad\textrm{for all} \ \alpha>0.
 \]
 \end{REM}

 \begin{prop}
 \label{prop:a_1 a_2}
 Let $0<\alpha_1<\alpha_2$ and $X_1,X_2\in\CX(0)$.
 \smallskip

 \noindent
 {\rm(a)} $d_{\alpha_1}(X_1,X_2)\le d_{\alpha_2}(X_1,X_2)$.
 \smallskip

 \noindent
 {\rm(b)}
 We cannot find a constant $C= C(\alpha_1,\alpha_1)<1$ such
 that for all $X_1,X_2\in \CX(0)$,
 $d_{\alpha_1}(X_1,X_2)\le C d_{\alpha_2}(X_1,X_2)$.
 \end{prop}
 \begin{proof}
 (a) is obvious. To see (b), it suffices to consider $X_1$ and $X_2$ with
 $\Pr(X_1=0)=\Pr(X_2=1)=1$. Then $d_\alpha(X_1,X_2)=1$ for every $\alpha>0$.
 \smallskip
 \end{proof}

 From (a) of the preceding proposition, $X_n\to_{\alpha_2} X$ implies
 $X_n\to_{\alpha_1} X$ for every $\alpha_1<\alpha_2$.
 In the sequel we shall show that for any $\alpha\geq 2$, the inequality
 $\dtv(X_n,X)\leq d_\alpha(X_n, X)$ holds true,
  provided
 $\{X,X_1,X_2,\ldots\}\subseteq \CX(1)$.
 To this end, we shall make use of the following {\it``moment inversion''}
 formula.
 \begin{lem}
 \label{lem:pmf via fm}
 If $X\in\CX(1)$ then its pmf $p$ can be written as
 \begin{equation}
 \label{eq:pmf via fm}
 p(j)=\sum_{k=j}^{\infty}\frac{(-1)^{k-j}}{k!}{k \choose j}\E(X)_k,
 \quad j=0,1,\ldots~.
 \end{equation}
 \end{lem}
 \begin{proof}
 By the assumption $X\in\CX(1)$, we can find a number $t'>1$ such that
 $\E(1+t')^X=\sum_{j=0}^{\infty}(1+t')^jp(j)<\infty$.
 Since $X$ is non-negative, its probability generating function admits a
 Taylor expansion around 0 with radius of convergence $R\ge1+t'>2$, i.e.,
 $P(u)=\sum_{j=0}^{\infty}u^jp(j)\in\RR$, $|u|<R$. It is well known that
 $\frac{\ud^k}{\ud u^k}P(u)\big|_{u=1}=\E(X)_k$, and since
 $P$ admits a Taylor expansion
 around 1 with radius of convergence $R'\ge t'>1$, we have
 \[
 P(u)=\sum_{k=0}^{\infty}\frac{\E(X)_k}{k!}(u-1)^k,\quad |u-1|<R'.
 \]
 Using the preceding expansion and the fact that $0\in(1-R',1+R')$ we get
 \[
 p(j)=\frac{1}{j!}\cdot\frac{\ud^j}{\ud u^j}P(u)\bigg|_{u=0}
     =\frac{1}{j!}\sum_{k=j}^{\infty}\frac{(u-1)^{k-j}}{(k-j)!}\E(X)_k\bigg|_{u=0}
     =\sum_{k=j}^{\infty}\frac{(-1)^{k-j}}{j!(k-j)!}\E(X)_k,
 \]
 completing the proof.
 \end{proof}

 It should be noted that the condition $X\in\CX(t)$ for some
 $t\in[0,1)$  is not sufficient for \eqref{eq:pmf via fm}.
 As an example, consider the geometric rv $X$ with pmf
 $p(j)=2^{-j-1}$, $j=0,1,\ldots$~.
 It is clear that $X\notin\CX(1)$, but $X\in\CX(t)$ for each $t\in[0,1)$. The
 factorial moments of $X$ are $\E(X)_k=k!$, $k=0,1,\ldots$, and the rhs of
 \eqref{eq:pmf via fm}, $\sum_{k=j}^{\infty}(-1)^{k-j}{k \choose j}$,
 is a non-convergent series.

 \begin{theo}
 \label{theo:fm2=>tv}
 If $X_1,X_2\in\CX(1)$ then $\dtv(X_1,X_2)\le d_\alpha(X_1,X_2)$ for each $\alpha\ge2$.
 \end{theo}
 \begin{proof}
 In view of Proposition \ref{prop:a_1 a_2}(a), it is enough to prove the
 desired result for $\alpha=2$. By \eqref{eq:rtv discrete} and
 \eqref{eq:pmf via fm} we get
 \[
 \dtv(X_1,X_2)
   =\frac{1}{2}\sum_{j=0}^{\infty}\left|\sum_{k=j}^{\infty}\frac{(-1)^{k-j}}{k!}
   {k \choose j}\big(\E(X_1)_k-\E(X_2)_k\big)\right|
   \le
   \frac{1}{2}\sum_{j=0}^{\infty}\sum_{k=j}^{\infty}\frac{1}{k!}
    {k \choose j}\big|\E(X_1)_k-\E(X_2)_k\big|.
 \]
 Interchanging the order of summation
 according to Tonelli's Theorem, we have
 \[
 \dtv(X_1,X_2)
   \le\frac{1}{2}\sum_{k=0}^{\infty}\sum_{j=0}^{k}\frac{1}{k!}
   {k \choose j}\big|\E(X_1)_k-\E(X_2)_k\big|
   =\sum_{k=0}^{\infty}\frac{2^{k-1}}{k!}\big|\E(X_1)_k-\E(X_2)_k\big|.
 \]
 The proof is completed by the fact that $\E(X_1)_0=\E(X_2)_0=1$.
 \smallskip
 \end{proof}

 Theorem \ref{theo:fm2=>tv} quantifies the fact that
 for any $\alpha\geq 2$, the $d_\alpha$ convergence (in $\CX(1)$)
 implies
 the convergence in total variation, and provides
 convenient bounds for the rate of the total variation
 convergence.
 However, we note that such convenient bounds
 do not hold for $\alpha < 2$.
 In fact, for given $\alpha\in(0,2)$ and $t\geq 0$,
 we cannot find a finite constant  $C=C(\alpha,t)>0$ such that
 $\dtv(X_1,X_2)\le C d_\alpha(X_1,X_2)$ for all $X_1,X_2\in\CX(t)$;
 see Remark \ref{rem3.1}, below.

 \section{An application to a generalized matching problem}
 \label{sec:matching,generalized}

 Consider the classical matching problem where, now,
 we record only a proportion of the matches, due to a random
 censoring mechanism.
 The censoring mechanism decides independently to every individual match.
 Specifically, when a particular match occurs,
 the mechanism counts this match with probability $\lambda$,
 independently of the other matches, and
 ignores this match with probability $1-\lambda$, where $0<\lambda\leq 1$.
 We are now interested on the number $Z_n(\lambda)$ of the counted matches.
 The case $\lambda=1$ corresponds to the classical matcing problem
 where all coincidences are recorded,  so that $Z_n=Z_n(1)$.

 The probabilistic formulation is as follows:
 Let $\bbb{\pi}_n=\big(\pi_n(1),\ldots,\pi_n(n)\big)$ be a random permutation
 of $\{1,\ldots, n\}$, as in the Introduction. Let also $J_1(\lambda),\ldots,J_n(\lambda)$
 be iid Bernoulli($\lambda$) rv's, independent of $\bbb{\pi}_n$.
 The number $Z_n(\lambda)$ of the recorded coincidences can be written as
 \[
 Z_n(\lambda)=\sum_{i=1}^n J_i(\lambda) \one\{\pi_n(i)=i\}.
 \]
 Let $A_i=\{J_i(\lambda)=1\}$, $B_i=\{\pi_n(i)=i\}$, $E_i=A_i\cap B_i$,
 $i=1,\ldots,n$. Then $Z_n(\lambda)$ presents the number of the
 events $E$'s that will occur and, by standard combinatorial arguments,
 \[
 \Pr(Z_n(\lambda)=j)=\Pr(\textrm{exactly }j \textrm{ among } E_1,\ldots,E_n \textrm{ occur})
 =\sum_{i=j}^n (-1)^{i-j} {i \choose j} S_{i,n},
 \]
 where
 \[
 S_{0,n}=1,
 \ \ \
 S_{i,n}=\sum_{1\leq k_1<\cdots<k_i\leq n} \Pr(E_{k_1}\cap \cdots \cap E_{k_i}),
 \ \ i=1,\ldots,n.
 \]
 Since the $A$'s are independent of the $B$'s, we have
 \[
 \Pr(E_{k_1}\cap \cdots \cap E_{k_i})=
 \Pr(A_{k_1}\cap \cdots \cap A_{k_i})
 \Pr(B_{k_1}\cap \cdots \cap B_{k_i}) = \lambda^i \frac{(n-i)!}{n!},
 \]
 so that
 \[
 S_{i,n}={n\choose i} \lambda^i \frac{(n-i)!}{n!}=\frac{\lambda^i}{i!},
 \ \ \ i=0,1,\ldots,n.
 \]
 Therefore, the pmf of $Z_n(\lambda)$ is given by
 \begin{equation}
 \label{3.1}
 p_{n;\lambda}(j):=\Pr(Z_n(\lambda)=j)
 =\frac{1}{j!}\sum_{i=j}^n (-1)^{i-j} \frac{\lambda^i}{(i-j)!}
 =\frac{\lambda^j}{j!}\sum_{i=0}^{n-j} \frac{(-\lambda)^{i}}{i!},
 \ \ j=0,1,\ldots,n~.
 \end{equation}
 The generalized matching distribution \eqref{3.1}
 has been introduced by
 \citet{Niermann1999},
 who showed that $p_{n;\lambda}$
 is a proper pmf for all $\lambda\in (0,1]$; however,
 Niermann did not give a probabilistic interpretation to the pmf
 $p_{n;\lambda}$, and
 derived its properties analytically.

 Since $\lim_{n\to\infty}\sum_{i=0}^{n-j}\frac{(-\lambda)^i}{i!}=e^{-\lambda}$
 for any fixed $j$, we see that $p_{n;\lambda}$ converges pointwise to the pmf of
 $Z(\lambda)$,  where $Z(\lambda)$ is a Poisson rv with mean $\lambda$,
 $\Poi(\lambda)$. Interestingly enough,
 the Poisson approximation is extremelly
 accurate; numerical results are shown in
 \citeapos{Niermann1999}
 work.
 Also, Niermann proved that $\E Z_n(\lambda)=\Var Z_n(\lambda)=\lambda$ for
 all  $n\geq 2$ and $\lambda\in(0,1]$.
 In fact, the following general result shows that the first $n$
 moments of $Z_n(\lambda)$ and $Z(\lambda)$ are identical,
 giving some light
 to the amazing accuracy of the Poisson approximation.

 \begin{lem}
 \label{lem.3.1}
 For any $\lambda\in(0,1]$, $\E\big(Z_n(\lambda)\big)_k=\lambda^k\one{\{k\le{n}\}}$,
 $k=1,2,\ldots$~.
 \end{lem}
 \begin{proof}
 For $k>n$ the relation is obvious, since $Z_n(\lambda)\in\CD_n$.
 For $k=1,2,\ldots,n-1$,
 \[
 \E\big(Z_n(\lambda)\big)_k
 =\sum_{j=k}^{n}\frac{\lambda^j}{(j-k)!}\sum_{i=0}^{n-j}\frac{(-\lambda)^i}{i!}
 =\lambda^k\sum_{r=0}^{n-k}\frac{\lambda^r}{r!}\sum_{i=0}^{(n-k)-r}\frac{(-\lambda)^i}{i!}
 =\lambda^k\sum_{r=0}^{n-k}p_{n-k;\lambda}(r),
 \]
 and, since
 $p_{n-k;\lambda}$ is a pmf supported on $\{0,1,\ldots,n-k\}$, we get the desired result.
 For $k=n$, $\E\big(Z_n(\lambda)\big)_n=n!p_{n;\lambda}(n)=\lambda^n$, completing the proof.
 \end{proof}

 \begin{cor}
 \label{cor:inf generalized}
 For any $\lambda\in(0,1]$ and $\alpha>0$,
 \[
 \inf_{X\in\CD_n}\big\{d_\alpha\big(X,Z(\lambda)\big)\big\}=d_\alpha\big(Z_n(\lambda),Z(\lambda)\big).
 \]
 \end{cor}

 Thus, for $\lambda\in(0,1]$,
 $Z_n(\lambda)$ minimizes the factorial moment distance from $Z(\lambda)$ over
 all rv's supported in a subset of $\{0,1,\ldots,n\}$. Using
 \eqref{eq:pmf via fm} it is easily verified that $Z_n(\lambda)$
 is unique.
 Moreover, it is worth pointing out that for $\lambda>1$, we cannot
 find a random variable $X\in \CD_n$ such that $\E (X)_k=\lambda^k\one\{k\leq n\}$
 for all $k$. Indeed, since $\CD_n\subset \CX(\infty)\subset \CX(1)$, assuming $X\in \CD_n$
 and $\E (X)_k=\lambda^k\one\{k\leq n\}$, we get from \eqref{eq:pmf via fm} that
 \[
 0\leq \Pr(X=n-1)=\frac{\lambda^{n-1}(1-\lambda)}{(n-1)!},
 \]
 which implies that $\lambda\leq 1$. Therefore, finding
 $\inf_{X\in\CD_n}\big\{d_\alpha\big(X,Z(\lambda)\big)\big\}$ for $\lambda>1$
 seems to be a rather difficult task.

 We now evaluate some exact and asymptotic results
 for the factorial moment distance and the total variation  distance
 between $Z_n(\lambda)$ and $Z(\lambda)$
 when
 $\lambda\in(0,1]$.

 \begin{theo}
 \label{theo:rfm generalized}
 Fix $\alpha>0$ and $\lambda\in(0,1]$ and let
 $d_\alpha(n):=d_\alpha\big(Z_n(\lambda),Z(\lambda)\big)$.
 Then,
 \begin{equation}
 \label{fm}
 d_{\alpha}(n)=\frac{\alpha^n \lambda^{n+1}}{n!}
 \int_0^1 (1-y)^n e^{\alpha\lambda y} \ud y.
 \end{equation}
 Moreover, the following double inequality holds:
 {
 \begin{equation}
 \label{eq:1<rfm(Z_n,Z)<e^e generalized}
 \frac{\alpha^{n}\lambda^{n+1}}{(n+1)!}
 \left(1+\frac{\alpha\lambda}{n+2}
 +\frac{a^2\lambda^2 }{(n+2)(n+3)}\right)
 <d_\alpha(n)
 <
 \frac{\alpha^{n}\lambda^{n+1}}{(n+1)!}
 \left(1+\frac{\alpha\lambda}{n+2}
 +\frac{a^2\lambda^2 e^{\alpha\lambda}}{(n+2)(n+3)}\right).
 \end{equation}
 }
 Hence, as $n\to\infty$,
 \begin{equation}
 \label{rate_fm}
  d_\alpha(n)\sim
  \frac{\alpha^n\lambda^{n+1}}{(n+1)!}
  \ \ \
  \textrm{and, more precisely}, \ \
  d_\alpha(n)
 =
 \frac{\alpha^{n}\lambda^{n+1}}{(n+1)!}
 \left(1+\frac{\alpha\lambda}{n+2}+o\left(\frac{1}{n}\right)
 \right).
 \end{equation}
 \end{theo}
 \begin{proof}
 From the definition of $d_\alpha$ and in view of \eqref{taylor}
 and Lemma
 \ref{lem.3.1},
 \[
 d_\alpha(n)=\frac{1}{\alpha}\sum_{k=n+1}^\infty \frac{(\alpha\lambda)^k}{k!}
 =\frac{1}{\alpha} \left(e^{\alpha\lambda}-\sum_{k=0}^n
 \frac{(\alpha\lambda)^k}{k!}\right)=\frac{1}{\alpha n!}
 \int_0^{\alpha\lambda} (\alpha\lambda-x)^n e^{x}\ud x,
 \]
 and the substitution $x=\alpha\lambda y$ leads to \eqref{fm}.
 Now \eqref{eq:1<rfm(Z_n,Z)<e^e generalized} follows from the inequalities
 $1+\alpha\lambda y +
 \frac{1}{2}\alpha^2\lambda^2 y^2<e^{\alpha\lambda y}
 <1+\alpha\lambda y +
 \frac{1}{2}e^{\alpha\lambda}\alpha^2\lambda^2 y^2$, $0<y<1$, while
 \eqref{rate_fm} is evident from \eqref{eq:1<rfm(Z_n,Z)<e^e generalized}.
 \end{proof}
 Theorems \ref{theo:fm2=>tv} and \ref{theo:rfm generalized} give the next
 \begin{cor}
 \label{cor:tv<e^2}
 An upper bound of $\dtv(Z_n,Z)$ is given by
 \begin{equation}
 \label{cor:tv<e^2}
 \dtv(Z_n,Z)<\frac{2^{n}}{(n+1)!}\left(1+\frac{2}{n+2}
 +\frac{4e^2}{(n+2)(n+3)}\right)\sim \frac{2^{n}}{(n+1)!}.
 \end{equation}
 \end{cor}

 The bound in \eqref{cor:tv<e^2} is of the correct order, and the
 same is true for the better result \eqref{eq:tv-DG}, given by
 \citet{DasGupta1999,DasGupta2005}.
 In contrast, the bound $\dtv(Z_n,Z)\leq\frac{2^n}{n!}$, given by
 \citet{Dia1987},
 is not assymptotically optimal, because
 $\frac{2^n}{(n+1)!}=o\left(\frac{2^n}{n!}\right)$.
 Thus, it is of some interest to point out
 that the factorial distance $d_2$ provides an optimal rate
 upper bound for the variational distance in the matching problem.
 The situation is similar for the generalized matching distribution, as
 the following result shows.

 \begin{theo}
 \label{theo:order tv}
 For any $\lambda\in(0,1]$, let
 $\dtv(n):=\dtv\big(Z_n(\lambda),Z(\lambda)\big)$
 be the variational distance
 between $Z_n(\lambda)$ and $Z(\lambda)$. Then,
 \begin{equation}
 \label{tv}
 \dtv(n)=\frac{\lambda^{n+1}}{2 n!}
 \int_0^1 [y^n+(2-y)^n] e^{-\lambda y} \ud y.
 \end{equation}
 Moreover, the following inequalities hold:
  \begin{equation}
 \label{inequality}
 \begin{split}
 \dtv(n)&>\frac{2^{n}\lambda^{n+1}}{(n+1)!}
 \left(1-\frac{2\lambda}{n+2}\left(1-\frac{1}{2^{n+1}}\right)\right);
 \\
 \dtv(n)
 &<
 \frac{2^{n}\lambda^{n+1}}{(n+1)!}
 \left(1-\frac{2\lambda}{n+2}\left(1-\frac{1}{2^{n+1}}\right)
 +\frac{4\lambda^2}{(n+2)(n+3)}\left(1-\frac{n+3}{2^{n+2}}\right)\right).
 \end{split}
 \end{equation}
 Hence, as $n\to\infty$,
 \begin{equation}
 \label{rate_tv}
  \dtv(n)\sim
  \frac{2^n\lambda^{n+1}}{(n+1)!}
  \ \ \
  \textrm{and, more precisely,} \ \
  \dtv(n)=
  \frac{2^n\lambda^{n+1}}{(n+1)!}\left(1-\frac{2\lambda}{n+2}
  +o\left(\frac{1}{n}\right)\right).
 \end{equation}
 \end{theo}
 \begin{proof}
 Clearly, \eqref{rate_tv} is an immediate consequence
 of \eqref{inequality}. Moreover, the inequalities
 \eqref{inequality} are obtained  from \eqref{tv} and the fact that
 $1-\lambda y <e^{-\lambda y}<1-\lambda y+\frac12 \lambda^2y^2$, $0<y<1$.
 It remains to show \eqref{tv}. From \eqref{eq:rtv discrete}
 with $p_1=p_{n;\lambda}$ and $p_2$ the pmf of $\Poi(\lambda)$, we get
 \[
 \dtv(n)
 =
 \sum_{j=0}^n \frac{\lambda^{j}}{j!} \left[
 \sum_{i=0}^{n-j}\frac{(-\lambda)^i}{i!}-e^{-\lambda} \right]^+
 =
 \sum_{j=0}^n \frac{\lambda^{j}}{j!} \left[ \frac{(-1)^{n-j}}{(n-j)!}
 \int_0^\lambda (\lambda-x)^{n-j} e^{-x} \ud x\right]^+,
  \]
 where the integral expansion is deduced by an application of \eqref{taylor}
 to the function $f(\lambda)=e^{-\lambda}$.
 Thus,
 \[
 \dtv(n)
 =
 \sum_{k=0}^n \frac{\lambda^{n-k}}{(n-k)!} \left[ \frac{(-1)^{k}}{k!}
 \int_0^\lambda (\lambda-x)^{k} e^{-x} \ud x\right]^+
 =
 \frac{1}{n!}
 \int_{0}^\lambda e^{-x} \left(\sum_{k \textrm{ even}} {n\choose k}
 (\lambda-x)^k \lambda^{n-k}\right)\ud x.
 \]
 Since
 \[
 \sum_{k \textrm{ even}} {n\choose k}
 (\lambda-x)^k \lambda^{n-k}=\frac{1}{2}\left[x^n+(2\lambda-x)^n\right],
 \]
 we obtain
 \[
 \dtv(n)
 =
 \frac{1}{2n!}
 \int_{0}^\lambda [x^n+(2\lambda-x)^n]e^{-x} \ud x,
 \]
 and a final change of variables $x=\lambda y$ yields
 \eqref{tv}.
 \end{proof}

 \begin{REM}
 \label{rem3.1}
 Althought the factorial moment distance $d_{\alpha}$
 dominates the variational distance when $\alpha\geq 2$,
 the situation for $\alpha\in(0,2)$ is completely different.
 To see this, assume that for some $\alpha\in(0,2)$ and  $t\ge0$
 we can find
 a finite constant $C=C(\alpha,t)>0$
 such that
 \begin{equation}
 \label{ineq}
 \dtv(X_1,X_2)\le C d_\alpha(X_1,X_2)\quad \textrm{for all} \ X_1,X_2\in\CX(t).
 \end{equation}
 Obviously, $Z$ and $Z_n$, $n=1,2,\ldots$~, lie $\CX(\infty)\subset\CX(t)$.
 From Theorem \ref{theo:order tv} we know that
 $\lim_{n\to\infty}\frac{(n+1)!}{2^n}\dtv(Z_n,Z)=1$.
 On the other hand, from \eqref{eq:1<rfm(Z_n,Z)<e^e generalized} with $\lambda=1$,
 \[
 d_\alpha(Z_n,Z)
 <
 \frac{\alpha^n}{(n+1)!}
 \left(1+\frac{\alpha}{n+2}
 +\frac{\alpha^{2}e^\alpha}{(n+2)(n+3)}
 \right),
 \]
 and,  since $|\alpha/2|<1$, this inequality
 contradicts \eqref{ineq}:
 \[
 1=\lim_{n\to\infty}\frac{(n+1)!}{2^n}\dtv(Z_n,Z)
 \le
 C\lim_{n\to\infty}\left(
 \left(\frac{\alpha}{2}\right)^{n}
 \left(1+\frac{\alpha}{n+2}
 +\frac{\alpha^{2}e^\alpha}{(n+2)(n+3)}
 \right)
 \right)=0.
 \]
 \end{REM}

 {
 
 }

\small
 \begin{thebibliography}{99}
 \small
 \bibitem
 [Barbour et al., 1992]
 {BHJ1992}
    {\sc Barbour, A.D., Holst, L.}\ and {\sc Janson, S.}\ (1992).
    \textit{Poisson Approximation.}
    Clarendon Press, Oxford University Press, NewYork.

 \bibitem
 [DasGupta, 1999]
 {DasGupta1999}
    {\sc DasGupta, A.}\ (1999).
    The matching problem with random decks and the Poisson approximation.
    \textit{Technical report, Purdue University.}

 \bibitem
 [DasGupta, 2005]
 {DasGupta2005}
    {\sc DasGupta, A.}\ (2005).
    The matching, birthday and the strong birthday problem: a contemporary review.
    \textit{J.\ Statist.\ Plann.\ Inference}, {\bf130}, 377--389.

 \bibitem
 [Diaconis, 1987]
 {Dia1987}
    {\sc Diaconis, P.}\ (1987).
    Application of the methods of moments in probability and statistics.
    \textit{Proc.\ Sympos.\ Appl.\ Math.}, {\bf37}, 125--142.

 \bibitem
 [Niermann, 1999]
 {Niermann1999}
    {\sc Niermann, S.}\ (1999).
    A generalization of the matching distribution.
    \textit{Statistical Papers},  {\bf40}(2), 233--238.

 \bibitem
 [Wang, 1991]
 {Wang1991}
    {\sc Wang, Y.H.}\ (1991).
    A compound Poisson convergence theorem.
    \textit{Ann.\ Probab.},  {\bf19}, 452--455.


 \end{thebibliography}
 \end{document}